\documentclass[11pt,oneside]{amsart}

\usepackage{amsmath,ifthen, amsfonts, amssymb,
srcltx,
   amsopn}

\usepackage{graphicx}

\newcommand{\showcomments}{yes}
\renewcommand{\showcomments}{no}

\newsavebox{\commentbox}
\newenvironment{com}%
{\ifthenelse{\equal{\showcomments}{yes}}%
{\footnotemark
        \begin{lrbox}{\commentbox}
        \begin{minipage}[t]{1.25in}\raggedright\sffamily\tiny
        \footnotemark[\arabic{footnote}]}
{\begin{lrbox}{\commentbox}}}%
{\ifthenelse{\equal{\showcomments}{yes}}%
{\end{minipage}\end{lrbox}\marginpar{\usebox{\commentbox}}}
{\end{lrbox}}}

\newtheorem{thm}{Theorem}[section]
\newtheorem{lem}[thm]{Lemma}

\newtheorem{cor}[thm]{Corollary}

\newtheorem{prop}[thm]{Proposition}

\theoremstyle{definition}

\newtheorem{rem}[thm]{Remark}

\DeclareMathOperator{\stabilizer}{Stabilizer}

\newcommand{\field}[1]{\mathbb{#1}}
\newcommand{\integers}{\ensuremath{\field{Z}}}

\newcommand{\hyperbolic}{\ensuremath{\field{H}}}

\newcommand{\boundary}   {{\ensuremath \partial}}

\begin{document}

\title{A boundary criterion for cubulation}
\author[N.~Bergeron]{Nicolas Bergeron}
\address{Institut de Math\'ematiques de Jussieu\\
Unit\'e Mixte de Recherche 7586 du CNRS\\
Universit\'e Pierre et Marie Curie\\
4, place Jussieu 75252 Paris Cedex 05, France}
\email{bergeron@math.jussieu.fr}

\author[D.~T.~Wise]{Daniel T. Wise}
           \address{Dept. of Math. \& Stats.\\
                    McGill University \\
                    Montreal, Quebec, Canada H3A 2K6 }
           \email{wise@math.mcgill.ca}
\date{\today}
\thanks{Research supported by NSERC}

\begin{abstract}
We give a criterion in terms of the boundary
for the existence of a proper cocompact  action of a word-hyperbolic group
on a CAT(0) cube complex.
We describe applications towards lattices and hyperbolic 3-manifold groups.
In particular, by combining the theory of special cube complexes, the surface subgroup result of Kahn-Markovic,
and Agol's criterion,
we find that every subgroup separable closed hyperbolic 3-manifold
is virtually fibered.
 \end{abstract}

\maketitle

\section{Introduction}
Let $G$ be a finitely generated group with Cayley graph $\Gamma$.
A subgroup $H\subset G$ is \emph{codimension-1} if it has a finite neighborhood $N_r(H)$
such that $\Gamma-N_r(H)$ contains at least two components that are \emph{deep}
in the sense that they do not lie in any $N_s(H)$. For instance any $\integers^n$ subgroup of $\integers^{n+1}$ is codimension-1,
and any infinite cyclic subgroup of a closed surface subgroup is as well.

Given a finite collection of codimension-1 subgroups $H_1,\dots, H_k$ of $G$, Michah Sageev introduced a simple but powerful
construction that yields an action of $G$ on a CAT(0) cube complex $C$ that is \emph{dual} to a system of walls associated to these subgroups
\cite{Sageev95}.

For each $i$, let $N_i=N_{r_i}(H_i)$ be a neighborhood of $H_i$ that separates $\Gamma$ into at least two deep components.
The \emph{wall} associated to $N_i$ is a fixed partition $\{\overleftarrow N_i, \overrightarrow N_i\}$
 consisting of one of these deep components $\overleftarrow N_i$ together with its complement $\overrightarrow N_i=\Gamma-\overleftarrow N_i$,
and more generally, the translated \emph{wall} associated to $gN_i$ is the partition $\{g\overleftarrow N_i, g\overrightarrow N_i\}$.
The two parts of the wall are \emph{halfspaces}.

We presume a certain degree of familiarity with the details of Sageev's construction here,
but hope that any interested reader will mostly be able to follow the arguments.
We shall not describe the structure of the dual cube complex $C$ here but will describe its 1-skeleton.
A $0$-cube of $C$ is a choice of one halfspace from each wall such that each element of $G$ lies in all but finitely many
of these chosen halfspaces. A wall is thought of as \emph{facing} the points in its chosen halfspace.
Two 0-cubes are joined by a 1-cube precisely when their choices differ on exactly one wall.

The walls in $\Gamma$ are in one-to-one correspondence with the hyperplanes of the CAT(0) cube complex $C$ given by Sageev's construction,
and the stabilizer of each such hyperplane equals the codimension-1 subgroup that stabilizes the associated translated wall:
The stabilizer of the hyperplane corresponding to a translated wall  associated to $gN_i$ is commensurable with $gH_i g^{-1}$.

Cocompactness properties of the resulting action were analyzed in \cite{Sageev97} where Sageev proved that:
\begin{prop}\label{prop:sageev compactness}
Let $G$ be a word-hyperbolic group, and $H_1,\dots, H_k$ be a collection of quasiconvex codimension-1 subgroups.
Then the action of $G$ on the dual cube complex is cocompact.
\end{prop}
We refer to \cite{HruskaWiseAxioms} for a more elaborate discussion of the finiteness properties of the action
obtained from Sageev's construction, as well as for background and an account of the literature.

In parallel to Proposition~\ref{prop:sageev compactness},
is a properness criterion which we state as follows (see for instance \cite{HruskaWiseAxioms}).
We use the notation $\#(p,q)$ for the number of walls separating $p,q$.
\begin{prop}\label{prop:linear separation}
If  $\#(1,g)\rightarrow \infty$ as $d_\Gamma(1,g)\rightarrow \infty$ then $G$ acts properly on $C$.
\end{prop}

An alternative to Proposition~\ref{prop:linear separation} is the following:
\begin{prop}\label{prop:axis separation}
Let $H_1,\dots, H_k$ be a collection of quasiconvex codimension-1 subgroups of
the word-hyperbolic group $G$.

Suppose that for each infinite order element $g$ of $G$,
there is a translate $fN_i$ such that
$g^{-n}fN_i$ and $g^nfN_i$ are separated by $fN_i$ for some $n$
in the sense that the corresponding partitions are nested:
$$f \overleftarrow N_i \subset g^{\pm n} f\overleftarrow N_i \ \ \ \mbox{ and } \ \ \ f \overrightarrow N_i \subset g^{\mp n} f\overrightarrow N_i .$$
Then $G$ acts properly on the dual CAT(0) cube complex.
\end{prop}
\begin{proof}[Sketch]
By Proposition~\ref{prop:sageev compactness}, $G$ acts cocompactly,
so it suffices to show that the stabilizer of each 0-cube of $C$ is finite.
If an infinite order element $g$ fixes a 0-cube $v$ of $C$ then $g^n$ would fix $v$ for each $n$.
But then the sequence $\{g^{nr}fN_i: r\in \integers\}$ is shifted by $g^n$ and so the walls in this sequence
would all face in the same direction.
By traveling in one or the other direction of this infinite sequence, we see that there are infinitely many walls
that do not face $1\in G$, which contradicts that $v$ is a 0-cube.
\end{proof}

We have found that in many cases it is difficult and sometimes quite messy to directly verify Proposition~\ref{prop:linear separation}~or~\ref{prop:axis separation}.
Moreover, in many cases, when there is a profusion of available codimension-1 subgroups
it is desirable to choose them in a flexible enough way so that there are sufficiently many to satisfy the properness criterion of
 Proposition~\ref{prop:axis separation}, while still keeping a finite number so that Proposition~\ref{prop:sageev compactness} for cocompactness is satisfied.

We propose the following criterion which is our main result:
\begin{thm}\label{thm:soft separation}
Let $G$ be word-hyperbolic.
Suppose that for each pair of distinct points $(u,v) \in (\boundary G)^2$
there exists a quasiconvex codimension-1 subgroup $H$
such that $u$ and $v$ lie in distinct
components of $\boundary G-\boundary H$.

Then there is a finite collection $H_1,\dots, H_k$ of quasiconvex codimension-1 subgroups
such that $G$ acts properly and cocompactly on the resulting dual CAT(0) cube complex.
\end{thm}
It is clear how the hypothesis of Theorem~\ref{thm:soft separation} relates
to Proposition~\ref{prop:axis separation}.
Indeed, let $g$ be an infinite order element in $G$. According to Theorem~\ref{thm:soft separation}
there exists a quasiconvex codimension-1 subgroup $H$
such that the attracting and repelling limit points $g^{\pm\infty}$  lie in distinct
components of $\boundary G-\boundary H$.
The iterates $g^n\boundary H$ form a nested sequence in $\boundary G$,
and hence so do the corresponding walls for large multiples of~$n$.

We prove Theorem~\ref{thm:soft separation} in \S\ref{sec:2}. The notion of virtual specialness is briefly recalled in \S\ref{sec:3}.

As an application we revisit our earlier work towards the cubulation of arithmetic lattices in \S\ref{sec:5}.
We also build upon a fundamental new result of Kahn-Markovic to see
that for every closed hyperbolic 3-manifold $M$, the group $\pi_1M$ is cubulated, see \S\ref{sec:6}.
It follows that if $\pi_1M$ is subgroup separable then $\pi_1M$ is virtually special
and hence, $M$ is virtually fibered by Agol's criterion. With some more work this holds under the milder assumption that
the quasi-fuchsian surface subgroups are separable. A much more elaborate proof that $\pi_1M$ is subgroup separable
and virtually special is given in \cite{WiseIsraelHierarchy} when $M$ has a geometrically finite incompressible surface. Finally we extend Theorem~\ref{thm:soft separation} to the relatively hyperbolic
setting in the last section.

Over the last years, lectures on special cube complexes have included the explanation
that the existence of sufficiently many separable surface subgroups implies
the virtual specialness of hyperbolic 3-manifold groups.
It is satisfying to record this in writing in view of the results of Kahn-Markovic.

Since a first version of this paper was circulated,
\begin{com} Perhaps we can remove the initial part of this sentence for the Archive and or final version.
If there is a reference to his thesis (i imagine there will be one before acceptance, then we add it here)\end{com}
 Guillaume Dufour has informed us that, motivated by the cubulation of compact hyperbolic 3-manifolds, he has
obtained a criterion similar to Theorem~\ref{thm:soft separation} in the case of cocompact lattices
of $\hyperbolic^n$. \begin{com} We should get a title from him \end{com}

\section{Soft Separation}\label{sec:2}
Let $X$ be a compact metrizable space. A group $G$ acts by homeomorphisms on $X$ as a \emph{convergence group}
if it acts properly discontinuously on the space of pairwise distinct triples in $X$ (see below).
Equivalently, $G$ acts as a convergence group
if for every sequence $(g_n)_n$ in $G$, there exists a subsequence $(g_{n_k})_k$ and $a,b$ in $X$
such that the sequence $(g_{n_k})_k$ converges uniformly on compact subsets of $X-\{a \}$
to the point $b$. The group $G$ is a \emph{uniform convergence group}
 if the action on the space of triples is also cocompact. In that case, $X$ is unique up to equivariant homeomorphism.
 If $G$ acts properly discontinuously on some locally compact Hausdorff space $Y$,
 then there is a unique compactification $Y\cup X$ of $Y$ that is natural for the action of $G$.

Bowditch proved that any properly discontinuous group action on a locally compact $\delta$-hyperbolic metric
space extends to a convergence group action on the boundary \cite{Bowditch99}.

Let $G$ be a word-hyperbolic group. Then $G$ acts properly discontinuously and cocompactly
on its cayley graph $\Gamma$ and we may identify the Gromov boundary $\boundary \Gamma$ with $\boundary G$
(see \cite{GhysBook90} for more details). Note that $X=\Gamma \cup \boundary \Gamma$ has a natural compact topology
and is metrizable but has no preferred metric. The group $G$ acts as
a convergence group on $X$ and as a uniform convergence
group on $\boundary G$, see \cite[Proposition 1.12 and 1.13]{Bowditch99}.
That this last property characterizes word-hyperbolic groups was a longstanding open problem until it was
resolved by Bowditch.

\emph{Triple space} is the subspace $T \subset (\boundary G)^3$ consisting of pairwise
distinct triples of points $\{ (x,y,z) \in (\boundary G)^3 \; : \; x \neq y \neq z \neq x \}$.
Since $G$ acts as a convergence group on $\boundary G$, it acts properly on $T$. Here the action
is the restriction of the diagonal action $g(x,y,z) = (gx, gy , gz)$ induced by the action of $G$ on
$\boundary G$. It easily follows that:

\begin{lem} \label{Properness}
For each neighborhood $M'$ of $(u,v,w)\in T$ there is a subneighborhood $M\subset M'$ containing $(u,v,w)$
such that $gM\cap M=M$ for some finite subgroup of $G$, and otherwise
$gM\cap M=\emptyset$.
\end{lem}
\begin{proof}
By properness we can find $M'' \subset M'$ such that the collection of elements
$K = \{ g\in G \; : \; (u,v,w)\in gM''\cap M'' \neq \emptyset \}$ is finite.
We can assume that $K$ is a subgroup, since otherwise we could obtain a smaller collection by using the new
neighborhood $M=\cap_{k\in K} kN''$.
Moreover, each element of $k$ maps $M$ to itself homeomorphically.
\end{proof}

\medskip

\begin{proof}[Proof of Theorem~\ref{thm:soft separation}]
For each pair of distinct points $u,v \in \boundary G$, by hypothesis there is a quasiconvex codimension-1 subgroup $H$
such that $u$ and $v$ are separated by $\boundary H \subset \boundary G$.
Let $U_H,V_H$ be the two regions of $\boundary G-\boundary H$ so $\boundary H$ separates $U_H,V_H$.
 We let $N$ be a finite neighborhood of $H$ such that
$\Gamma - N$ contains at least two deep components.
 Note that $\boundary N$ equals $\boundary H$.
  Let $\overrightarrow N$ be the component of $\Gamma-N$ such that $u \in \boundary \overrightarrow N$.
  Let $\overleftarrow N$ equal $\Gamma-\overleftarrow N$.
For each $j\in G$ we define the translated wall associated to $jH$
to be the partition $\{j \overrightarrow N , j \overleftarrow N\}$.

Let\footnote{Note that $T$ is empty when $G$ is elementary
in which case the theorem is trivial.} $w \in \boundary G - \{u,v\}$. By Hausdorffness of $\boundary G$,
there exists an open neighborhood $W'$ of $w$ that is disjoint from open neighborhoods of
$u$ and $v$.
We let $U'=U_H-\overline{W'}$ and $V'=V_H-\overline{W'}$, so that $M=U' \times V ' \times W' \subset T$.
Then $M'=U'\times V'\times W'$ is an open neighborhood of $(u,v,w)$ such that
$U',V'$ are separated by $\boundary H$. By Lemma~\ref{Properness} we can refine this to an open neighborhood $M=U\times V\times W$ of $(u,v,w)$ such that:
\begin{enumerate}
\item $U$ and $V$ are separated by $\boundary H$
\item $gM\cap M=M$ for those $g$ in some finite subgroup of $G$, and otherwise $gM\cap M=\emptyset$.
\end{enumerate}
Observe that $GM$ is saturated with respect to the quotient $T\rightarrow \bar T = G \backslash T$.

The above construction applies to each point $(u,v,w) \in T$
and we shall use the explicit notation $M_{(u,v,w)}$ for the neighborhood $M$ chosen above.

Consider the following collection of open saturated neighborhoods:
\begin{eqnarray} \label{collection1}
\big\{ \cup _{ j\in G} jM_{(u,v,w)} \; : \;   (u,v,w) \in T \big\}.
\end{eqnarray}
The collection forms an open covering of $T$.

Since $G$ acts as a uniform convergence group on $\boundary G$ the quotient
$\bar T$ is compact. Since each element of collection~(\ref{collection1}) is saturated with respect to $T\rightarrow \bar T$,
the compactness of $\bar T$ assures that there is a finite subcollection that also covers $T$.

For each infinite order $g\in G$ consider a point $(g^{+\infty}, g^{-\infty}, w') \in T$.
We have shown that $(g^{+\infty}, g^{-\infty}, w')$ lies in one of the sets of our finite subcollection,
so in particular, $(g^{+\infty}, g^{-\infty}, w')$ must lie in some $jM_{(u,v,w)}$ associated to some point
$(u,v,w) \in T$ as above,
where $M_{(u,v,w)}=U_{(u,v,w)}\times V_{(u,v,w)}\times W_{(u,v,w)}$
and where the associated quasiconvex codimension-1 subgroup $H_{(u,v,w)}$ has the property that $\boundary H_{(u,v,w)}$ separates $U_{(u,v,w)}$ from $V_{(u,v,w)}$.
Thus  $g^{+\infty}\in jU_{(u,v,w)}$ and $g^{-\infty}\in jV_{(u,v,w)}$ and these are separated by
$j\boundary H_{(u,v,w)}$.

Now $G$ acts as a convergence group on $X$ and $N_{(u,v,w)} \cup \boundary H_{(u,v,w)}$ is a compact subset of $X-\{ j g^{-\infty} \}$. This implies that $jg^nj^{-1}$ converges to $jg^{+\infty}$
uniformly on $N_{(u,v,w)} \cup \boundary H_{(u,v,w)}$. Thus the element $jgj^{-1}$ is
``separated'' by the wall associated to $N_{(u,v,w)}$, equivalently the element $g$ is separated by
$jN_{(u,v,w)}$. We note that $\overrightarrow N_{(u,v,w)}$
and $\overleftarrow N_{(u,v,w)}$ are both deep since they respectively contain
$g^{+\infty}$ and $g^{-\infty}$ in their boundaries.
\end{proof}

\begin{com}The following cautionary remark prepares the reader for the nonuniform case.
 Connectedness at infinity is surprisingly not used in the uniform case.
This suggests that there may be a slightly better approach in the proof of Theorem~\ref{Thm:last}
perhaps avoiding the dichotomy between loxodromic elements that stabilize or don't stabilize a component of the boundary.

The remark unifies (but doesn't really enlarge) the range of applications.\end{com}
\begin{rem}\label{rem:emptyset separation}
When $H$ is a finite codimension-1 subgroup of $G$, then $\boundary H=\emptyset$,
and $\boundary G$ is already disconnected.
We can regard $\boundary G$ as being ``separated'' by $\boundary H$
by letting the two parts consist of $\boundary \overleftarrow N$ and $\boundary \overrightarrow N$,
where $\Gamma = \overleftarrow N \sqcup \overrightarrow N$ is a partition into a pair of deep $H$-invariant subspaces.
We will revisit this point in \S\ref{sec:nonuniform case} where connectivity of $\boundary G$
arises in the nonuniform generalization of Theorem~\ref{thm:soft separation}.
\end{rem}

\section{Virtual Specialness}\label{sec:3}
A nonpositively curved cube complex is \emph{special} if it admits a local isometry to
the cube complex associated with a right-angled artin group.
Special cube complexes were introduced in \cite{HaglundWiseSpecial} where they were initially defined
in terms of illegal configurations of immersed hyperplanes.
Recall that a \emph{hyperplane} in a CAT(0) cube complex is
a connected separating subspace consisting
of a maximally extending sequence of ``midcubes'' each of which cuts its ambient cube in half.
For each hyperplane $\widetilde D \subset \widetilde C$
in a CAT(0) cube complex $\widetilde C$, one obtains an \emph{immersed hyperplane} $D\rightarrow C$
where $D=\stabilizer_{\pi_1C}(\widetilde D) \backslash \widetilde D$.

Some of the most important properties of a special cube complex are \cite{HaglundWiseSpecial}:
\begin{prop}\label{prop:special properties}
Let $C$ be a special cube complex. Then
\begin{enumerate}
\item $\pi_1C$ embeds in a right-angled artin group.
\item Thus $\pi_1C$ is residually torsion-free nilpotent.
\item If $C$ is compact and $\pi_1C$ is word-hyperbolic then
every quasiconvex subgroup $H$ of $\pi_1C$ is separable.
\end{enumerate}
\end{prop}

\begin{rem} \label{R1}
The proof of Proposition~\ref{prop:special properties}.(3) produces a finite index subgroup $V$ of $\pi_1 C$ that \emph{retracts}
onto $H$, see \cite[Theorem 7.3]{HaglundWiseSpecial}. This means that $H$ is contained in $V$ and  there is a homomorphism $\rho :V \rightarrow H$ whose restriction to $H$ is the identity.
\end{rem}

A cube complex is \emph{virtually special} if it has a special finite cover. Likewise
a group $G$ is \emph{virtually [compact] special} if $G$ has a finite index subgroup that acts freely [and cocompactly]
 on a CAT(0) cube complex $C$ with special quotient.
 \footnote{We shall more generally say that a group $G$ \emph{virtually} has a property
P if there is a finite index subgroup
of $G$ having property P.}
Not every nonpositively curved cube complex is virtually special,
but the following criteria hold:
\begin{prop}\label{prop:virtually special criteria}
\begin{enumerate}
\item \label{criterion:double} A nonpositively curved cube complex with finitely many immersed hyperplanes is virtually special
if and only if $\pi_1D\pi_1E$ is separable for each pair of crossing immersed hyperplanes in $C$.
\item \label{criterion:quasiconvex}  A compact nonpositively curved cube complex with word-hyperbolic $\pi_1$ is virtually special if each quasiconvex subgroup is separable.
\item \label{criterion:single} A compact nonpositively curved cube complex with word-hyperbolic $\pi_1$ is virtually special if $\pi_1D$ is separable for each immersed hyperplane.
\item \label{criterion:embedded} A compact nonpositively curved cube complex with word-hyperbolic $\pi_1$ is virtually special if each immersed hyperplane is embedded.
\end{enumerate}
\end{prop}
Note that Criteria~(\ref{criterion:double})~and~(\ref{criterion:quasiconvex}) were obtained in \cite{HaglundWiseSpecial},
Criterion~(\ref{criterion:single}) is a more difficult criterion established in \cite{HaglundWiseAmalgams},
and Criterion~(\ref{criterion:embedded}) is a considerably deeper result (depending also on the proof of Criterion~(\ref{criterion:single}))
that is established in \cite{WiseIsraelHierarchy}.

\section{Cubulating Arithmetic Hyperbolic Lattices}
\label{sec:5}
In \cite{BergeronHaglundWiseSimple}~and~\cite{HaglundWiseAmalgams} we cubulated certain uniform hyperbolic lattices $G$
by showing that there is a finite family of regular hyperplanes $H_1,\dots, H_k$
such that the complement of their translates $\hyperbolic^d- GH_i$ consists of uniformly bounded pieces.
A similar approach works in the nonuniform case.
\footnote{Since the hyperplanes are regular it is then possible to deduce virtual specialness by one of several
options, e.g. double separablity (in general), or single separability and a virtually malnormal hierarchy (in the uniform case).}

\begin{thm}\label{thm:soft lattice cubulation}
Let $G$ be a uniform arithmetic hyperbolic lattice with a codimension-1 quasiconvex subgroup $H$.
Then $G$ acts properly and cocompactly on a CAT(0) cube complex.
\end{thm}
\begin{proof}
Let $U,V$ be the components of $\boundary G$ separated by $\boundary H$.
For any two distinct points $p,q \in \boundary G$ there is an element $c$ of the commensurator of $G$
such that $cp\in U$ and $cq\in V$.
Note that $H_c = c^{-1}Hc \cap G$ is of finite index in $c^{-1}Hc$, and is thus itself a codimension-1 quasiconvex subgroup.
Observe that $\boundary H_c = c^{-1}\boundary H$ separates $p,q$
since $p\in c^{-1}U$ and $q\in c^{-1}V$.
We have thus satisfied the criterion of Theorem~\ref{thm:soft separation}.
\end{proof}

Let $F$ be a totally real number field, $(V,q)$ be a quadratic space over $F$,
of signature $(n,1)$ at one place and definite at all other places. Arithmetic lattices
$G \subset {\rm SO} (V , F)$ are of \emph{simple type}. Such a lattice is uniform
if and only if $(V,q)$ is anisotropic over $F$.

\begin{thm}\label{thm:soft simple lattice special}
Let $G$ be a uniform arithmetic hyperbolic lattice of simple type.
Then $G$ is virtually special.
\end{thm}
\begin{proof}
We follow the proof of Theorem~\ref{thm:soft lattice cubulation} using the codimension-1 subgroup
associated to a regular hyperplane.
The result is a proper cocompact action on a CAT(0) cube complex.
The fundamental groups of immersed hyperplanes of $\bar C = G\backslash C$
correspond to the stabilizers of the chosen regular hyperplanes.
The double coset separability criterion for virtual specialness thus holds by \cite[Prop~10 and proof of Thm~6]{Bergeron2002}.
See \cite{BergeronHaglundWiseSimple} for more details.

\end{proof}

\begin{rem}[Slight generalizations]
If $G$ is an arithmetic hyperbolic lattice with a separable quasiconvex codimension-1 subgroup
then the analogous proof shows that $G$ is virtually special, but Proposition~\ref{prop:virtually special criteria}.(\ref{criterion:single})
must be used.
Likewise, if $G$ is an arithmetic hyperbolic lattice with a quasiconvex codimension-1 subgroup
that doesn't self-cross its translates (on the boundary) then we obtain a quasiconvex hierarchy for $G$
which is thus virtually special using Proposition~\ref{prop:virtually special criteria}.(\ref{criterion:embedded}).
\end{rem}

We will give a nonuniform version of Theorem~\ref{thm:soft simple lattice special} in
holds as well.  Its proof uses Theorem~\ref{Thm:last} instead of Theorem~\ref{thm:soft separation}.

\begin{rem}[Complex hyperbolic lattices cannot be compactly cubulated] \label{rem:DG}
Delzant and Gromov proved that complex hyperbolic lattices
cannot have codimension-1 quasiconvex subgroups \cite{DelzantGromov05}.
Nevertheless, complex hyperbolic arithmetic lattices of simple type have positive virtual first Betti
number (see \cite{BorelWallach80}).
This yields codimension-1 subgroups. Most\footnote{In fact all of them except in dimension
$3$ and $7$.} real hyperbolic arithmetic lattices (even those that are not of simple type) embed as quasiconvex
subgroups of complex hyperbolic arithmetic lattices of simple type, see e.g. \cite{Lubotzky96}.
It is thus of considerable interest to know whether the codimension-1 subgroups in simple complex hyperbolic arithmetic lattices
yield quasiconvex codimension-1 subgroups of the corresponding real hyperbolic arithmetic lattices.
\end{rem}

\section{Cubulating Hyperbolic 3-manifolds}\label{sec:6}
Jeremy Kahn and Vladimir Markovic obtained the outstanding result that every hyperbolic 3-manifold
contains a closed quasi-fuchsian immersed surface \cite{KahnMarkovic09}.
In fact, they proved the following powerful:

\begin{prop}\label{prop:KM}
Let $M$ be a closed hyperbolic 3-manifold, and regard $\pi_1M$ as acting on $\hyperbolic^3\cong \widetilde M^3$.
For each great circle $C\subset \boundary \hyperbolic^3$ there is a sequence of immersed quasi-fuchsian surfaces
$F_i\rightarrow M$ such that $\boundary \widetilde F_i$ pointwise converges to $C$.
\end{prop}

An immediate consequence is that:
\begin{cor}\label{cor:circle separation}
Let $M$ be a closed hyperbolic 3-manifold.
For each pair of distinct points $p,q\in \boundary \widetilde M$
there is an immersed quasi-fuchsian surface $F\rightarrow M$ with a lift of universal cover $\widetilde F\subset \widetilde M$
such that $\boundary \widetilde F$ separates $p,q$ in $\boundary \widetilde M$.
\end{cor}
\begin{proof}
Let $C$ be the great circle that is the perpendicular bisector of a geodesic from $p$ to $q$.
By Proposition~\ref{prop:KM}, let $F_i\rightarrow M$ be a sequence of surfaces whose universal covers
have boundaries that limit to $C$.

For sufficiently large $i$, the $\epsilon$-neighborhood of $C$ in $\boundary \widetilde M$
has the property that $\boundary \widetilde F_i$ separates its two bounding circles (just like $C$).
Thus $\boundary \widetilde F_i$ separates $p,q$ in $\boundary \widetilde M$.
\end{proof}

By combining Corollary~\ref{cor:circle separation} with Theorem~\ref{thm:soft separation}
and Proposition~\ref{prop:sageev compactness},
we obtain the following result:
\begin{thm}\label{thm:cubulating 3-manifolds}
Let $M$ be a closed hyperbolic 3-manifold.
Then $\pi_1M$ acts freely and cocompactly on a CAT(0) cube complex.
\end{thm}

\begin{rem}
The proof of Theorem~\ref{thm:cubulating 3-manifolds} depends crucially on Proposition~\ref{prop:KM}.
 We note that Lackenby proved that
any arithmetic 3-manifold contains a surface group \cite{Lackenby08}. It follows from his proof that when the $3$-manifold is
closed the surface subgroup he constructs
is quasiconvex. Theorem~\ref{thm:cubulating 3-manifolds} for closed arithmetic $3$-manifolds
thus follows from the combination of Lackenby's theorem and Theorem~\ref{thm:soft lattice cubulation}.
\end{rem}

Let $M$ be a closed hyperbolic 3-manifold $M$.
The celebrated ``virtual Haken conjecture" for $M$
would follow from the result of Kahn and Markovic and the other well-known conjecture that:
\begin{eqnarray} \label{conj}
\mbox{\emph{quasi-fuchsian surface subgroups of $\pi_1 M$ are separable}},
\end{eqnarray}
which means that quasi-fuchsian surface subgroups are closed in the profinite topology
of $\pi_1 M$. In fact it follows from Proposition~\ref{prop:virtually special criteria}.(1) and
the word-hyperbolicity of $\pi_1 M$ that if moreover all quasiconvex subgroup are closed in the profinite topology then $\pi_1 M$ is virtually special.

Agol shows even more: if $\pi_1 M$ is virtually special then $M$ is virtually fibered \cite{Agol08}. Specifically,
he introduces the condition \emph{residually finite rational solvable} (RFRS) and proves the following:

\begin{prop}  Let $M$ be a closed $3$-manifold. If $\pi_1 M$ is
RFRS then $M$ is virtually fibered.
\end{prop}

The RFRS condition is a generalization of being poly-free-abelian that holds for many residually torsion-free nilpotent groups.
As right-angled Artin groups are residually torsion-free nilpotent (unpublished \begin{com} actually we can reference his PhD Thesis. For instance...\end{com}work of Droms),
by \cite[Thm.~1.5]{HaglundWiseSpecial} special groups are also residually torsion-free nilpotent.
In fact - as observed in \cite{WiseIsraelHierarchy} - the Droms variant of the Magnus representation
shows that right-angled Artin groups satisfy RFRS. Alternatively Agol directly proves that right-angled Coxeter groups are RFRS, and this virtually yielded the condition for right-angled Artin groups \cite[Thm.~2.2]{Agol08}.

Using Proposition~\ref{prop:virtually special criteria}.(3) - shown in \cite{HaglundWiseAmalgams} -
we obtain the following result:
\begin{thm}
If quasi-fuchsian surface subgroups of a hyperbolic 3-manifold $M$ are separable,
then $M$ is virtually fibered.
\end{thm}
\begin{proof}
By Theorem~\ref{thm:cubulating 3-manifolds}, $\pi_1 M = \pi_1C$ where $C$ is a compact nonpositively curved cube complex.
As surface subgroups are separable, we apply Proposition~\ref{prop:virtually special criteria}.(3)
to obtain a finite special cover of $C$ with $\widehat C$.
As recalled above this implies that $\pi_1 C$ virtually statisfies the (RFRS) condition of Agol.
Consequently, the corresponding finite cover $\widehat M$ satisfies Agol's criterion,
and so $\widehat M$ has a finite cover that fibers.
\end{proof}
According to \cite{WiseIsraelHierarchy} it even suffices for $\pi_1M$ to have a single separable
quasi-fuchsian subgroup to obtain virtual specialness and hence virtual fibering.

\section{Relatively Hyperbolic Extension}
\label{sec:nonuniform case}
In this section we generalize Theorem~\ref{thm:soft separation} to a relatively hyperbolic
situation. The notion of a relatively hyperbolic group was introduced by Gromov and
has been developed by various authors. We mainly follow Brian Bowditch's and Asli Yaman's treatments as developed
in \cite{Bowditch99RH, Yaman}. Let $G$ be a
finitely generated group with Cayley graph $\Gamma$. A {\it peripheral structure} on $G$ consists of a set $\mathcal{G}$ of infinite subgroups of $G$ such that each $P \in \mathcal{G}$ is equal to its normalizer in $G$, and each $G$-conjugate of $P$ lies in $\mathcal{G}$. We refer to an element of $\mathcal{G}$ as a {\it maximal parabolic} subgroup.

Bowditch \cite[\S 4]{Bowditch99RH} defined the notion
of a {\it hyperbolic $G$-set}. Such a set contains only many $G$-orbits and it moreover corresponds to it a quasi-isometry class of connected $G$-invariant Gromov-hyperbolic graphs.

The set $\mathcal{G}$ is a $G$-set. Let us assume that it contains only finitely many distinct $G$-conjugacy classes and let $P_1 , \ldots , P_s$ be a choice of representatives so that $\mathcal{G} = \{
P_i^g \; : \; i=1 , \ldots , s, \ g \in G \}$.

$G$ is {\it hyperbolic relative to $P_1 , \ldots , P_s$} if either $s=0$ and
$G$ is hyperbolic, or if $s\neq 0$ and $\mathcal{G}$ is a hyperbolic $G$-set.
We let $\overline{\Gamma}$ be a corresponding Gromov-hyperbolic graph so that the
maximal parabolic subgroups correspond to the infinite vertex stabilizers.

The boundary of $G$ is defined as the union $V_{\infty} (\overline{\Gamma}) \cup \boundary \overline{\Gamma}$, where $\boundary \overline{\Gamma}$ is the Gromov boundary of $\overline{\Gamma}$ and $V_{\infty} (\overline{\Gamma})$
is the vertex set of infinite valence in $\overline{\Gamma}$. It has a natural topology as a compact Hausdorff space
(see e.g. \cite[p. 37]{Yaman}).
This space, which we denote by $\boundary G$, only depends on $G$ and $P_1,
\ldots, P_s$.

The group $G$ acts as a convergence group on $\boundary G$.
It thus follows from the dynamical
definition of a convergence group that each infinite order element $g \in G$ has either one or two
fixpoints in $\boundary G$. The element $g$ is {\it loxodromic} if it has two fixpoints.
The conjugates of $P_1, \ldots, P_s$, are precisely the maximal \emph{infinite} subgroups that fix a
point $p \in \boundary G$ (maximal parabolic subgroups).
Such a subgroup acts properly and discontinuously on $\boundary G - \{ p\}$
and the quotient ${\rm Stab}_G (p) \backslash (\boundary G-\{p \})$ is compact.
A {\it parabolic point} is the fixpoint of a conjugate of one of $P_1, \dots, P_s$.

Let $T$ continue to denote the triple space of $\boundary G$.
The quotient $\bar T = G \backslash T$ is the union of a compact set and finitely many quotients
of ``cusp neighborhoods'' of parabolic points.

The triple space $T$ can be compactified by adding a copy of $\boundary G$, see \cite{Bowditch99}.
We thus topologize $T \cup \boundary G$ so that a
sequence $(u_i , v_i , w_i)$ in $T$ converges to $x$ in $T \cup \boundary G$ if and only if at least two
of the sequences $(u_i )$, $(v_i )$ and $(w_i )$ converge to $x$ in $\boundary G$.

A {\it cusp neighborhood} of a parabolic point $p \in \partial G$ is defined as follows:
Let $C$ be a compact subset of $\boundary G-\{p\}$ such that ${\rm Stab}_G (p) C = \boundary G - \{ p \}$ and let $K$ be a compact
subset containing an open
neighborhood of $C$ in $T \cup \boundary G$. A cusp neighborhood of $p$ in $T$
is an open subset of the form $T- {\rm Stab}_G (p) K$.

Let $p_1, \ldots , p_s$ be representatives from the $G$-classes of parabolic points in $\boundary G$
such that $P_i = {\rm Stab}_G (p_i)$.
The quotient $\bar T$ becomes a compact Hausdorff space upon the addition of the points $Gp_1, \ldots , Gp_s$,
 so that a neighborhood basis of $Gp_i$
is given by
 $\{Gp_i\} \cup G\backslash B$ where $B$ ranges over cusp neighborhoods of $p_i$. In particular we may
choose an open cusp neighborhood $B_i$ of $p_i$,  $i=1 , \ldots , s$, such that the $gB_i$, $g\in G/P_i$, $i=1 , \ldots , n$, are pairwise disjoint and
\begin{eqnarray} \label{Tth}
T_{\rm thick} := T - \bigcup_{i=1}^n \bigcup_{g \in G} gB_i
\end{eqnarray}
projects onto a compact subspace of $\bar T$.

Given a subgroup $H$ of $G$ we let $\boundary H$ be the set of
all accumulation points of $H$-orbits in $\boundary G$.
Note that if $N$ is a finite neighborhood of
$H$ in $\Gamma$ we may define $\boundary N \subset \boundary G$ as the intersection of
$\boundary G$ with the closure of any orbit of $N$ in $\bar \Gamma$.
Thus $\boundary N$ equals $\boundary H$.

Theorem~\ref{thm:soft separation} now generalizes as follows:
\begin{thm} \label{Thm:last}
Let $G$ be hyperbolic relative to a collection of maximal parabolic subgroups $P_1 , \ldots , P_s$.
Suppose that for each pair of distinct points $u,v \in \boundary G$ there exists a
quasi-isometrically embedded codimension-1 subgroup
$H$ such that $u$ and $v$ lie in distinct components of $\boundary G - \boundary H$.

Suppose that for each parabolic point $p$, there exist finitely many quasi-isometrically embedded
codimension-1 subgroups of $G$ whose intersections with ${\rm Stab}_G (p)$ yield
a proper action of ${\rm Stab}_G (p)$ on the corresponding dual cube complex.

Then there exist finitely many quasi-isometrically embedded codimension-1 subgroups of
$G$ such that the action of $G$ on the dual cube complex is proper.
\end{thm}

Before proving Theorem~\ref{Thm:last} we first need to generalize Proposition~\ref{prop:axis separation} to a relatively hyperbolic situation.
This is the subject of the following subsection.

\subsection{Relatively Hyperbolic Axis Separation}
\label{sec:r.h. axis separation}
While Proposition~\ref{prop:axis separation} immediately generalizes under the assumption that $C$ is locally finite,
this is itself not always easily deduced.
In fact, it is not always the case that $C$ is locally finite - even in the relatively hyperbolic case.
For instance, let $G$ be the free product $P_1*P_2$ where each $P_i$ is the fundamental group of a locally-infinite
cube complex, and note that $G$ is hyperbolic relative to $P_1,P_2$.

Let $G$ be hyperbolic relative to a collection of parabolic subgroups $P_1,\dots, P_s$.
Given a finite collection of codimension-1 subgroups $H_1,\dots, H_k$ of $G$, we let
$N_i$ be as in the introduction and let $W_i$ denote the wall $\{\overleftarrow N_i, \overrightarrow N_i\}$.

The main theorem in \cite{HruskaWiseAxioms} can be stated as follows:
\begin{prop}\label{prop:sparse cubulation}
Let $G$ be a f.g. group that is hyperbolic relative to a collection of parabolic subgroups $P_1,\dots, P_s$.
Let $H_1,\dots, H_k$ be a collection of quasi-isometrically embedded codimension-1 subgroups of $G$.
Let $C$ denote the CAT(0) cube complex dual to the $G$-translates of $W_1,\dots, W_k$.
For each $i$, let $C_i$ denote the CAT(0) cube complex dual to the walls in $P_i$ corresponding to the nontrivial
walls obtained from intersections with translates of the $W_i$, and note that $C_i$ embeds in $C$ as a convex subcomplex.
Then: \begin{enumerate}
\item there exists a compact subcomplex $K$ such that $C= GK \cup_{i=1}^s GC_i$, and
\item $g_iC_i \cap g_jC_j \subset GK$ unless $i=j$ and $g_j^{-1}g_i \in \stabilizer(C_i)$.
\end{enumerate}
\end{prop}

\begin{rem} \label{rem:connected}
We may further assume that $GK$ is connected.
Indeed, since $G$ is finitely generated, and $C$ is connected, one can add a collection of paths $S_i$ to $K$
such that each $S_i$ starts at the basepoint in $K$ and ends at the translate of this basepoint by the $i$-th generator of $G$.
\end{rem}

We shall now state a properness criterion. Under the hypothesis of Proposition~\ref{prop:sparse cubulation} we say that an infinite order element $g$ of $G$ satisfies the {\it axis separation condition} if there
is a translate $fW_i$ such that the walls
$g^{-n} fW_i$ and $g^{n} f W_i$ are separated by $fW_i$ for some $n$,
in the sense that the partitions are nested.

\begin{lem}\label{lem:r.h. axes separation}
Let $G$ be hyperbolic relative to $P_1,\dots, P_s$.
Let $H_1,\dots, H_k$ be a finite collection of quasi-isometrically embedded codimension-1 subgroups of $G$ that are associated to walls $W_1,\dots, W_k$.
Then $G$ acts properly on the dual cube complex $C$ provided the following two conditions both hold:
\begin{enumerate}
\item Each $P_i$ acts properly on $C$. More explicity, $P_i$ acts properly on and stabilizes
a convex subcomplex $C_i\subset C$. Moreover, $C_i$ is isomorphic to the cubulation of $P_i$
 induced by intersections with the relevant $G$-translates of the walls.
\item The axis separation condition holds for loxodromic elements.
\end{enumerate}
\end{lem}
Before proceeding to the main part of the proof of Lemma~\ref{lem:r.h. axes separation}
we prove the following partial result.
We follow the notation of Proposition~\ref{prop:sparse cubulation}
and assume that
$GK$ is connected by Remark~\ref{rem:connected}.

\begin{lem}\label{lem:GK proper}
The group $G$ acts properly on $GK$.
\end{lem}
\begin{proof}
Since $G$ acts cocompactly on $GK$, it suffices to show that no vertex of $GK$ has infinite stabilizer.

 Let $v$ be a vertex of $GK$.
We first prove that $\stabilizer(v)$ cannot contain a loxodromic element. Assume to the contrary that
some loxodromic element $g$ stabilizes $v$. Since the axis separation condition holds for $g$, there is a translate
$fW_i$ and an integer $n$ such that the walls
$g^{-n} fW_i$ and $g^{n} f W_i$ are separated by $fW_i$. The sequence
$\{ g^{nr} f W_i \; : \; r \in {\Bbb Z} \}$ of disjoint separating walls is then shifted by $g^n$.
As $g$ fixes $v$ we conclude that the walls
$g^{nr} W_i$ have to face the same direction in $v$. This violates the property of a vertex
(or rather 0-cube)
that all but finitely many walls face $1 \in G$.

Since $\stabilizer(v)$ contains no loxodromic element, it is either elliptic and hence finite,
or it is contained in some parabolic subgroup $P=hP_ih^{-1}$. By hypothesis, $P$ acts properly on its stabilized subcomplex $hC_i$. As both $v$ and the convex subcomplex
 $hC_i$ are stabilized, the point $v'$ in $hC_i$ that is closest to $v$ is also stabilized by $\stabilizer (v)$.
But $\stabilizer(v')$ is finite, so $\stabilizer(v)$ is finite as well.
\end{proof}

\begin{proof}[Proof of Lemma~\ref{lem:r.h. axes separation}]
Let $B$ be a finite ball in $C$. Then $B\cap GK$ lies in a finite ball $B'$ of $GK$.
By Lemma~\ref{lem:GK proper}, $G$ acts properly and cocompactly on $GK$, and so we see that $B'$ intersects
finitely many parabolic subgroups - or rather $B'$ intersects finitely many sets of the form $P_i^gK$.
Thus $B$ lies in the union of $B'$ and finitely many $hC_i$ that intersect $B'$.
\footnote{We note that the radius of $B'$ in $GK$ could be significantly larger than the
 radius of $B$ in $C$, as the map $GK\subset C$ might not be a quasi-isometric embedding.}
  For each $i \in\{1, \ldots, s\}$
we let $J_i$ denote the finite subset of $G$
consisting of elements $j$ with $jC_i \cap B'\neq \emptyset$.

We may now conclude the proof of the Proposition.
Suppose to the contrary that the set
$S=\{ g \in G \; : \; gB\cap B\neq \emptyset \}$ is infinite.
There are finitely many values of $g$ such that $gB'\cap B'\neq \emptyset$,
thus finitely many values where $g(B\cap GK)\cap B = g(B\cap GK)\cap (B\cap GK)$ is nonempty.

It hence follows from Proposition~\ref{prop:sparse cubulation}.(2) that for some $i$ and $j_1,j_2\in J_i$,
there are infinitely many $g \in S$ such that
 $g(j_1 C_i-GK)\cap (j_2 C_i-GK) \neq \emptyset$.
There are finitely many  $jC_i$ that intersect $B'$, so again Proposition~\ref{prop:sparse cubulation}.(2)
implies that there exists $i$ and $j\in J_i$ and infinitely many $g\in S$ such that
$g(j C_i-GK)\cap (j C_i-GK) \neq \emptyset$. So $j^{-1}g j\in \stabilizer C_i$ for
 infinitely many values of $g\in S$. Note that $P_i$ - being a maximal
\begin{com} We assume that the parabolic subgroups $P_i$ are maximal parabolic subgroups, and hence equal their own commensurators...
\end{com}
parabolic subgroup - is not commensurated by any larger subgroup so
 $\stabilizer C_i = P_i$. We thus have proved
that $S \cap jP_i j^{-1}$ is infinite. Furthermore, if $g \in S$ we have
$gB\cap B\neq \emptyset$ and the translation length of $g$ in $C$
is bounded by $2r$ where $r$ is the radius of the ball $B$. We conclude that there are infinitely many
elements $g \in jP_i j^{-1}$ of translation length bounded by $2r$. This contradicts (the hypothesis)
that the parabolic subgroup $j P_i j^{-1}$ acts properly on its cubulation $jC_i$.
\end{proof}

\subsection{Proof of Theorem \ref{Thm:last}}

The relative accessability result \cite[Prop~10.2]{Bowditch99RH} states
that $G$ has a splitting as a graph of groups whose edge groups are finite and each of whose vertex groups
does not split nontrivially over any finite subgroup relative to peripheral subgroups.
Using the Bass-Serre tree of this accessability splitting as a guide, for each such finite edge group $E_i$,
we regard $\boundary E_i =\emptyset$ as separating $\boundary G$ according to the way a corresponding edge of the Bass-Serre tree
separates its boundary (see Remark~\ref{rem:emptyset separation}).
It follows from \cite[Prop~10.1]{Bowditch99RH} that for each vertex group $V_i$ of the accessability splitting,
$\boundary V_i$ is a connected component of $\boundary G$.

The proof will now be similar to the proof of Theorem~\ref{thm:soft separation}. The main difference is that
the quotient $\bar T = G \backslash T $ is no longer compact. Instead we use the decomposition
in Equation~(\ref{Tth}) and the following:

\begin{lem} \label{Lem:w}
Let $u,v$  be distinct non-parabolic points in the same component of $\boundary G$.
There exists $w \in \boundary G$ such that
$$(u,v,w) \in T_{\rm thick}.$$
\end{lem}
\begin{com} This proves that $w$ lies in the same component as $u,v$.Also, perhaps the Lemma holds without assuming non-parabolic.\end{com}

\begin{proof}
Let $L$ be the component of $\boundary G$ containing $u$ and $v$.
Let $\phi(L)$ denote the image of $L$ in $T \cup \boundary G$ under the continuous map
$\phi(w)=(u,v,w)$.
Since $\phi(L)$ is connected and Hausdorff and contains $\phi(u)$ and $\phi(v)$,
it must contain a third point $\phi(w)$ that lies in $T$.

If $\phi(w)$ doesn't lie in any cusp neighborhood then we are done, so let $B$ denote a cusp neighborhood containing $\phi(w)$.
 Observe that the closure of $B$ contains a single boundary point and thus contains neither $\phi(u)$ nor $\phi(v)$

Since $T\cup\boundary G$ is a compact Hausdorff space we may
choose $U$, $V$ to be open neighborhoods of $\phi(u),\phi(v)$ that are also disjoint from the closure of $B$.
Now, suppose that $\phi(L)$ is covered by $U,V$ and a collection of cusp neighborhoods,
and observe that this collection must include $B$.
Let $A$ denote the union of $U,V$ and all cusp neighborhoods besides $B$.
Then $A,B$ provides a separation of $\phi(L)$ - which is impossible.
\end{proof}

We may now conclude as in the proof of Theorem~\ref{thm:soft separation}.

For each pair of distinct points $u,v \in \boundary G$, by hypothesis there is a quasi-isometrically
embedded codimension-1 subgroup $H$
such that $u$ and $v$ are separated by $\boundary H \subset \boundary G$.
Let $U_H,V_H$ be the two corresponding regions of $\boundary G-\boundary H$.
For $j\in G$ we define the translated wall associated to $jH$
as in the proof of Theorem~\ref{thm:soft separation}.

Let $w \neq u,v$ and define $M_{(u,v,w)}$ as before. Consider the following collection
of open saturated neighborhoods:
\begin{eqnarray} \label{collection2}
\big\{ \cup _{ j\in G} jM_{(u,v,w)} \; : \;   (u,v,w) \in T_{\rm thick} \big\}.
\end{eqnarray}
This collection forms an open covering of $T_{\rm thick}$.

As the sets in collection~(\ref{collection2}) are saturated relative to $T\rightarrow \bar T$,
the compactness of the projection of $T_{\rm thick}$ to
$\bar T$ assures that there is a finite subcollection that still cover
$T_{\rm thick}$.

Consider an infinite order element $g\in G$. If $g$ is loxodromic then $g$ has two
distinct fixpoints $g^{\pm\infty}$  in $\boundary G$.
If $g^{\pm\infty}$ lie in distinct components of $\boundary G$
then the ``boundary'' of a conjugate of one of the finite edge groups from the accessability splitting
will function to separate them below. So let us consider the case where $g^{\pm\infty}$ lie in the same component.

Since $g^{\pm\infty}$ are not parabolic points, Lemma~\ref{Lem:w} provides a point
 $w' \in \boundary G$ such that:
$$(g^{+\infty}, g^{-\infty}, w') \in T - \bigcup_{i=1}^n \bigcup_{g \in G} gB_i.$$
We have shown that $(g^{+\infty}, g^{-\infty}, w')$ lies in one of the sets of our finite subcollection,
so in particular, $(g^{+\infty}, g^{-\infty}, w')$ must lie in some $jM_{(u,v,w)}$ associated to some point
$(u,v,w) \in T$ as above,
where $M_{(u,v,w)}=U_{(u,v,w)}\times V_{(u,v,w)}\times W_{(u,v,w)}$
and where the associated quasi-isometrically embedded codimension-1 subgroup $H_{(u,v,w)}$
has the property that $\boundary ( H_{(u,v,w)})$ separates $U_{(u,v,w)}$ from $V_{(u,v,w)}$.
Thus  $g^{+\infty}\in jU_{(u,v,w)}$ and $g^{-\infty}\in jV_{(u,v,w)}$ and these are separated by
$j\boundary H_{(u,v,w)}$. Now note that $G$ acts as a convergence group on $T\cup \boundary G$
and that considering the $G$-orbit of a point in $T_{\rm thick}$ yields a map from $\Gamma$
to $T$ which is proper when restricted to the axis of $g$. As in the proof of Theorem~\ref{thm:soft separation}
we thus conclude that the element $g$ is ``separated'' by the wall associated to $jN_{(u,v,w)}$.

This proves that the axis separation condition holds for the two types of loxodromic elements.
\begin{com}
We shall add all finite accessability edge groups to our finite collection (in addition to those we choose above using compactness).
So suppose they are in the same component. \end{com}
 By hypothesis any maximal parabolic
subgroup $P \subset G$ acts properly on the cube complex dual to the finitely many codimension-1 subgroups of $G$ whose intersections with $P$
properly cubulate it. Taking into account all these codimension-1 subgroups of $G$ we get a finite collection to which Lemma~\ref{lem:r.h. axes separation}
applies so that $G$ acts properly on the associated dual cube complex.

\bibliographystyle{alpha}
\bibliography{C:/papers/wise}

%
%
\end{document}